\documentclass[11pt]{article}
\usepackage{url}
\usepackage[margin=1.3in]{geometry}
\usepackage{cite}
\usepackage{multirow}
\usepackage{graphicx}
\usepackage{setspace}
\usepackage{algorithm}
\usepackage{algorithmic}

\usepackage{amsmath,amssymb,amsfonts}
 \usepackage{amsthm}
\usepackage{epsfig}
\usepackage{amsxtra}
\usepackage{amsbsy}
\usepackage{amsopn}
\usepackage{epsf}
\usepackage{subfigure}
\usepackage{psfrag}
\usepackage{color}
\usepackage[usenames,dvipsnames]{xcolor}
\usepackage{tikz} 
\usepackage{currfile}
\usepackage{filehook}

\def\d{\delta}
\def\r{\rho}
\def\e{\epsilon}

\def\om{\omega}
\def\a{\alpha}

\def\s{\sigma}

\def\gm{\gamma}
\def\Gm{\Gamma}
\def\A{\mathbf{A}}

\def\RR{\mathbb{R}}

\newcommand{\edges}{E}
\newcommand{\fail}{\mathcal F}
\newcommand{\support}{\mathcal S}
\newcommand{\suppest}{\widetilde{\mathcal S}}

\newcommand{\obs}{\mathbf y}
\newcommand{\x}{\mathbf x}
\newcommand{\z}{\mathbf z}
\newcommand{\sol}{\widehat{\x}}
\newcommand{\w}{\boldsymbol{\omega}}

\newcommand{\dimensions}{n \times N}
\newcommand{\numsam}{n}
\newcommand{\noise}{\mathbf e}

\newcommand{\bigO}{\mathcal O}

\newtheorem{theorem}{Theorem}[section]
\newtheorem{remark}{Remark}[section]

\newtheorem{definition}{Definition}[section]
\newtheorem{lemma}{Lemma}[section]

\begin{document}

\title{Sparse matrices for weighted sparse recovery}

\author{
Bubacarr Bah 
\thanks{
Mathematics Department, University of Texas at Austin ({\tt bah@math.utexas.edu}).
}
}

\maketitle

\begin{abstract}
We derived the first sparse recovery guarantees for weighted $\ell_1$ minimization with sparse random matrices and the class of weighted sparse signals, using a weighted versions of the null space property to derive these guarantees. These sparse matrices from expender graphs can be applied very fast and have other better computational complexities than their dense counterparts. In addition we show that, using such sparse matrices, weighted sparse recovery with weighted $\ell_1$ minimization leads to sample complexities that are linear in the weighted sparsity of the signal and these sampling rates can be smaller than those of standard sparse recovery. Moreover, these results reduce to known results in standard sparse recovery and sparse recovery with prior information and the results are supported by numerical experiments.
\end{abstract}


\section{Introduction} \label{sec:intro}

Weighted sparse recovery was introduced by \cite{rauhut2015interpolation} for application to function interpolation but there has been a steady growth in interest in the area, \cite{peng2014weighted,adcock2015infinite, adcock2015infinite2,bouchot2015compressed,bah2015sample,flinth2015optimal}, since then. In particular \cite{bah2015sample} showed its equivalence to standard sparse recovery using weighted $\ell_1$ minimization. The setting considered here is as follows: for a signal of interest $\x\in\RR^N$ which is $k$-sparse and a given measurement matrix $\A\in\RR^{\dimensions}$, we perform measurements $\obs = \A\x + \noise$ for a noise vector $\noise$ with $\|\noise\|_1 \leq \eta$ (this being the main difference with the previously mentioned works).  The weighted $\ell_1$-minimization problem is then given by:
\begin{equation}
\label{eqn:wl1min}\tag{WL1}
\min_{\z\in\RR^N} ~\|\z\|_{\om,1} \quad \mbox{subject to} \quad \|\A\z - \obs\|_1 \leq \eta,
\end{equation}
where $\|\z\|_{\om,1} = \sum_{i=1}^N \om_i|z_i|$ with weights $\om_i > 0$. If, based on our prior knowledge of the signal, we assign an appropriate weight vector $\w\geq 1$, then any $\support\subset [N]$ such that $\sum_{i\in \support} \om_i \leq s$ is said to be {\em weighted $s$-sparse}. When the weighted $\ell_1$ minimization \eqref{eqn:wl1min} is used to recover signals supported on sets that are weighted $s$-sparse, then the problem becomes a weighted sparse recovery problem.

The motivation for the use of weighted $\ell_1$ minimization either for standard sparse recovery or weighted sparse recovery stems from the numerous applications in which one naturally is faced with a sparse recovery problem and a prior distribution over the support. Consequently, this has been the subject of several works in the area of compressive sensing, see for example \cite{von2007compressed,candes2008enhancing,khajehnejad2009weighted,vaswani2010modified,jacques2010short,xu2010breaking,friedlander2012recovering,oymak2012recovery,rauhut2015interpolation,mansour2014recovery,peng2014weighted,adcock2015infinite,bouchot2015compressed,bah2015sample}. Most of these works proposed various strategies for setting weights (either from two fixed levels  as in \cite{khajehnejad2009weighted,mansour2011weighted}, or from a continuous interval as in  \cite{candes2008enhancing,rauhut2015interpolation}), in which smaller weights are assigned to those indices which are deemed ``more likely" to belong to the true underlying support. 

All of the prior work concerning weighted $\ell_1$ minimization considered dense matrices, either subgaussian sensing matrices or structured random matrices arising from orthonormal systems. Here, we instead study the weighted $\ell_1$ minimization problem using sparse binary matrices for $\A$. In the unweighted setting, binary matrices are known to possess what is referred to as the {\em square-root bottleneck}, that is they require $m= \Omega\left(k^2\right)$ rows instead of the {\em optimal} $\bigO\left(k\log\left(N/k\right)\right)$ rows to be ``good'' compressed sensing matrices with respect to optimal recovery guarantees in the $\ell_2$ norm, see \cite{devore2007deterministic,chandar2008negative}.  Yet, in \cite{berinde2008combining}, the authors show that such sparse matrices achieve optimal sample complexity ({\em optimally} few rows of  $\bigO\left(k\log\left(N/k\right)\right)$) if one instead considers error guarantees in the $\ell_1$ norm.  This paper aims to develop comparable results for sparse binary matrices in the setting of weighted $\ell_1$ minimization.
On the contrary to being second best in terms of theoretical guarantees to their dense counterparts, sparse binary matrices have superior computational properties. Their application, storage and generation complexities are much smaller than dense matrices, see Table \ref{tab:sampcomp}.
\begin{table}[h!]
\centering 
\begin{tabular}{|l|c|c|c|c|}
	\hline
 	 \small{\bf Ensemble} & \small{\bf Storage} & \small{\bf Generation} & \small{\bf Application} ($\A^*\obs$) & \small{\bf Sampling rate}\\
 	\hline 
	&  & &  & \\
    	 Gaussian \cite{candes2006stable} &  $\bigO(nN)$ & $\bigO(nN)$ &  $\bigO(nN)$ & $\bigO(k\log(N/k))$  \\
	&  & &  & \\
	Partial Fourier \cite{haviv2015restricted} & $\bigO(n)$ & $\bigO(n)$ &  $\bigO(N\log N)$ & $\bigO(k\log^2(k)\log(N))$  \\
	&  & &  & \\
	Expander \cite{berinde2008combining} & $\bigO(dN)$ & $\bigO(dN)$ &  $\bigO(dN)$ & $\bigO(k\log(N/k))$  \\
	&  & &  & \\
	\hline
\end{tabular}
\label{tab:sampcomp}
\caption{\small Computational complexities of matrix ensembles. The number of nonzeros per column of $\A$ is denoted by $d$, which is typically $\bigO(\log N)$. The references point to papers that proposed the sampling rate. The table is a slight modification of Table I in \cite{mendozaexpander}.}
\end{table} \vspace{-1mm}
Moreover, these non-mean zero binary matrices are more natural to use for some applications of compressed sensing than the dense mean zero subgaussian matrices, for example the single pixel camera, \cite{duarte2008single}, uses measurement devices with binary sensors that inherently correspond to binary and sparse inner products. 

The paper \cite{rauhut2015interpolation} introduced a general set of conditions on a sampling matrix and underlying signal class  which they use to provide recovery guarantees for weighted $\ell_1$ minimization, namely, the concepts of weighted sparsity, weighted null space property, and weighted restricted isometry property.  These are generalizations of the by-now classical concepts of sparsity, null space property, and restricted isometry property introduced in \cite{candes2006stable} for studying unweighted $\ell_1$ minimization.  The paper  \cite{rauhut2015interpolation}  focused on applying these tools to matrices arising from bounded orthonormal systems, and touched briefly on implications for dense random matrices.   Here, we show that, under appropriate  modifications, the same tools can provide weighted $\ell_1$ minimization guarantees for sparse binary sensing matrices that are adjacency matrices of expander graphs.

\paragraph{\bf Contributions:}
The contributions of this work are as follows.
\begin{enumerate}
 \item The introduction of the weighted robust null space property, satisfied by adjacency matrices of $(k,d,\e)$-lossless expander graphs, see Definition \ref{def:rnsp1} in Section \ref{sec:guarantees}.
 \item The characterization of weighted sparse recovery guarantees for \eqref{eqn:wl1min} using these matrices, see Theorem \ref{thm:errorbnd}, in Section \ref{sec:guarantees}.  
 \item The derivation of sampling rates that are linear in the weighted sparsity of the signals using such matrices, see Theorem \ref{thm:sampcomp_new} in Section \ref{sec:sample}. These sampling bounds recover known bounds for unweighted sparse recovery and sparse recovery with partially known support, see Section \ref{sec:disc}.
 \end{enumerate}
Numerical experiments support the theoretical results, see Section \ref{sec:numerics}.

\section{Preliminaries} \label{sec:prelim}

\subsection{Notation \& definitions}
Scalars will be denoted by lowercase letters (e.g. $k$), vectors by lowercase boldface letters (e.g., ${\bf x}$), sets by uppercase calligraphic letters (e.g., $\mathcal{S}$) and matrices by uppercase boldface letters (e.g. ${\bf A}$).
The cardinality of a set $\mathcal{S}$ is denoted by $|\mathcal{S}|$ and $[N] := \{1, \ldots, N\}$.
Given $\mathcal{S} \subseteq [N]$, its complement is denoted by $\mathcal{S}^c := [N] \setminus \mathcal{S}$ and $\x_\mathcal{S}$ is the restriction of $\x \in \RR^N$ to $\mathcal{S}$, i.e.~ $(\x_\mathcal{S})_i = \x_i$ if $i \in \mathcal{S}$ and $0$ otherwise. $\Gamma(\support)$ denotes the set of {\em neighbors} of $\support$, that is the right nodes that are connected to the left nodes in $\support$ in a bipartite graph, and $e_{ij} = (x_i,y_j)$ represents an edge connecting node $x_i$ to node $y_j$.
The $\ell_p$ norm of a vector ${\bf x} \in \RR^N$ is defined as $\|{\bf x}\|_p := \left ( \sum_{i=1}^N x_i^p \right )^{1/p}$, while
the weighted $\ell_p$ norm is $\|\x\|_{\om,p}:= \left(\sum_{i=1}^N \om_i^{2-p}|x_i|^p \right)^{1/p}$. This work focuses on the case where $p=1$, i.e $\|\x\|_{\om,1}:= \sum_{i=1}^N \om_i|x_i|$.

A $(k,d,\e)$-lossless expander graph, also called an {\em unbalanced} expander graph \cite{berinde2008combining}, is maximally  ``well-connected" given a fixed number of edges.  More precisely, it is defined as follows:

\begin{definition}[Lossless Expander]
\label{def:llexpander}
Let $G=\left( [N],[n],\edges \right)$ be a left-regular bipartite graph with $N$ left (variable) nodes, $m$ right (check) nodes, a set of edges $\edges$ and left degree $d$.
If, for any $\epsilon \in (0,1/2)$ and any $\support \subset [N]$ of size $|\support|\leq k$, we have that $|\Gamma(\support)| \geq (1-\epsilon)d|\support|$, then $G$ is referred to as a {\em $(k,d,\epsilon)$-lossless expander graph}.
\end{definition}

\subsection{Weighted sparsity}
As a model signal class for weighted $\ell_1$ minimization, we consider the weighted $\ell_p$ spaces considered in \cite{rauhut2015interpolation}. Given a vector of interest $\x\in\RR^N$ and a vector of weights $\w\in\RR^N\geq 1$, i.e. $\om_i \geq 1$ for $i\in [N]$, the weighted $\ell_p$ space is defined as 
\begin{equation}
 \label{eqn:wlpspace}
  \ell_{\om,p} := \left\{ \x\in\RR^N: ~\|\x\|_{\om,p} <\infty\right\}, \quad 0<p\leq 2.
 \end{equation}
 
 The weighted $\ell_0$-``norm'' is denoted as $\|\cdot\|_{\om,0}$; while the weighted cardinality of a set $\support$ is denoted as $\om(\support)$ and both are respectively defined as
 \begin{equation}
 \label{eqn:l0norm}
  \|\x\|_{\om,0} := \sum_{\{i:x_i\neq0\}} \om_i^2, \quad \mbox{and} \quad
  \om(\support) := \sum_{i\in\support} \om_i^2.
 \end{equation}
 Observe that the weighted cardinality upper bounds the cardinality of a set, i.e. $\om(\support) \geq |\support|$ since $\om_i\geq 1$. We denote the weighted $s$-term approximation of a vector $\x$ for $s\geq 1$ by $\s_s(\x)_{\om,p}$ and define it as follows:
  \begin{equation}
 \label{eqn:stermapprox}
  \s_s(\x)_{\om,p} := \inf_{\z:\|\z\|_{\om,0}\leq s} \|\z - \x\|_{\om,p}.
 \end{equation}
Up to a small multiplicative constant, this quantity can be computed efficiently by sorting a weighted version of the signal $\x$ and truncating, see \cite{rauhut2015interpolation} for more details.   For this work, we focus attention to the case $p=1$:
 \begin{equation}
 \label{eqn:stermapprox1}
  \s_s(\x)_{\om,1} := \|\x - \x_{\support}\|_{\om,1} = \|\x_{\support^c}\|_{\om,1}.
 \end{equation}
\vspace{-0.6cm}

\section{Theoretical results} \label{sec:mresults}

The main results of this work give recovery guarantees for weighted $\ell_1$ minimization \eqref{eqn:wl1min} when the sampling operators are adjacency matrices of lossless expander graphs for the class of weighted sparse signals.  We characterize the appropriate weighted robust null space property and expansion condition that the adjacency matrix of a $(k,d,\e)$-lossless expander graph needs to satisfy to guarantee robust weighted sparse recovery. These results reduce to the standard sparsity and unweighted $\ell_1$ minimization results when the weights are all chosen to be equal to one. We derive  sample complexities, in terms of the weighted sparsity $s$, of weighted sparse recovery using weighted $\ell_1$ minimization compared to unweighted $\ell_1$ minimization with adjacency matrices of a $(k,d,\e)$-lossless expander graphs. These sample complexities are linear in $\om(\support)$ and reduce to known results of standard sparse recovery and sparse recovery with prior information.

\pagebreak
\subsection{Robust weighted sparse recovery guarantees}\label{sec:guarantees}
The {\em weighted null space property} ($\w$-NSP) has been used to give sparse recovery guarantees \cite{khajehnejad2009weighted,mansour2011weighted,rauhut2015interpolation,mansour2014recovery} with two schemes for choice of weights. In \cite{khajehnejad2009weighted,rauhut2015interpolation} the weights $\w\geq 1$; whilst in \cite{mansour2011weighted,mansour2014recovery} the weights $\w\leq 1$. Similar to \cite{rauhut2015interpolation}, we consider the {\em weighted robust NSP} ($\w$-RNSP) for the type of matrices we focus on, which is the robust version of the NSP in the weighted case and follows from the unweighted RNSP proposed in \cite{foucart2013mathematical} for such matrices.
\begin{definition}[$\w$-RNSP]
\label{def:rnsp1}
 Given a weight vector $\w$, a matrix $\A\in \RR^{\dimensions}$ is said to have the robust $\w$-RNSP of order $s$ with constants $\r < 1$ and $\tau > 0$, if 
 \vspace{-1mm}\begin{equation}
 \label{eqn:rnsp}
  \|{\bf v}_{\mathcal{S}}\|_{\om,1} \leq \r\|{\bf v}_{\mathcal{S}^c}\|_{\om,1} + \tau\sqrt{s}\|\A{\bf v}\|_{1},
 \end{equation}
  for all ${\bf v}\in\RR^N$ and all $\mathcal{S}\subset [N]$ with $\omega(\mathcal{S}) \leq s$. 
\end{definition}

We will derive conditions under which an expander matrix satisfies the $\w$-RNSP to deduce error guarantees for weighted $\ell_1$ minimization \eqref{eqn:wl1min}. This is formalized in the following theorem.
\begin{theorem}
 \label{thm:wrnsp}
 Let $\A\in\{0,1\}^{\dimensions}$ be the adjacency matrix of a $(k,d,\e)$-lossless expander graph. If $\mathop{\e_{2k}<1/6}$, then $\A$ satisfies the $\w$-RNSP \eqref{eqn:rnsp} with
 \begin{align}
  \label{eqn:wrnspconst}
  \r = \frac{2\e_{2k}}{1-4\e_{2k}}, \quad \mbox{and} \quad
  \tau = \frac{1}{\sqrt{d}(1-4\e_{2k})}.
 \end{align}
\end{theorem}

\begin{proof}
 Given any $\z\in\RR^N$. Let $\w\in\RR^N$ be an associated weights vector with $\om_i\geq 1$, for $i\in [N]$, and $\support$ with $|\support|\leq k$ be such that $\om(\support)\leq s$. We will prove that if $\A$ is the adjacency matrix of a $(k,d,\e)$-lossless expander, then $\A$ will satisfy the $\w$-RNSP \eqref{eqn:rnsp} with the parameters specified in \eqref{eqn:wrnspconst}. 
 
 Given $\support$ is the index set of the $k$ largest in magnitude entries of $\z$ and let the indexes in $\support^c$ be ordered such that 
 \begin{equation}
  \label{eqn:orderingcondition}
  \om_i |z_i| \geq \om_{i+1} |z_{i+1}| \quad \mbox{for} \quad i\in\support^c\,.
 \end{equation}
  Without lost of generality (w.l.o.g) we assume that the set of variable nodes of the bipartite graph corresponding to the $(k,d,\e)$-lossless expander are ordered accordingly.

 We denote the {\em collision set} of edges of the bipartite graph as $E'$ and define it as thus
 \begin{equation}
  \label{eqn:collisionset}
  E' := \{e_{ij} ~| ~l<i \mbox{ such that } e_{lj}\in E \}\,.
 \end{equation}
We first state and prove the following lemma that will later be used in the proof.
\begin{lemma}
 \label{lem:collisionbound}
 Let $\x\in\RR^N$ be $k$-sparse. Then
 \begin{equation}
  \label{eqn:collisionbound}
  \sum_{e_{ij}\in E'} \om_i |x_i| \leq \e d \|\x\|_{\om,1}\,.
 \end{equation}
\end{lemma}
\begin{proof}
  Define $R_i := \{e_{lj}\in E' ~| ~l\leq i\}$ and $r_i = |R_i|$. Note that $r_0 = r_1 = 0$.
  \begin{align}
   \label{eqn:cbproof1}
   \sum_{e_{ij}\in E'} \om_i |x_i| 	&= \sum_{i = 1}^N \om_i |x_i| \left(r_i - r_{i-1}\right) \\
   \label{eqn:cbproof2}
					&\leq \sum_{i \leq k} \om_i |x_i| \left(\e di - \e d(i-1)\right) \\
  \label{eqn:cbproof3}
					&= \e d \|\x\|_{\om,1}\,.
  \end{align}
 Equation \eqref{eqn:cbproof1} comes from the definition of $R$. The restriction to $k$ indexes in \eqref{eqn:cbproof2} is due to the fact that $\x$ is $k$-sparse and $x_i = 0$ for all $i\in [N]\backslash \mbox{supp}(\x)$. The bound in \eqref{eqn:cbproof2} is due to the expansion property of the $(k,d,\e)$-lossless expander graph, which implies $r_{k'} \leq \e dk'$ for any $k'\leq k$.
\end{proof}

\pagebreak

 Continuing with the main proof, we split the index set of $\z$ into $q+1$ subsets of $[N]$ (i.e. $\support_0, \support_1, \ldots, \support_q$) of equal cardinality $k$ except possibly the last subset, $\support_q$. We also assume that $\om(\support_0) \leq s$. Let $\edges(\support)$ denote the set of edges connecting to nodes in $\support$.
 \begin{align}
 \label{eqn:wrnsp1}
 d\|\z_{\support_0}\|_{\om,1} = d\sum_{i \in\support_0} \om_i |z_i| & = \sum_{e_{ij} \in E(\support_0)} \om_i |z_i| \\
  \label{eqn:wrnsp1b}
 & = \mathop{\sum \om_i |z_i|}_{e_{ij} \in E(\support_0)\backslash E'(\support_0)}  + \sum_{e_{ij} \in E'(\support_0)} \om_i |z_i|\\
  \label{eqn:wrnsp2}
 													& = \sum_{j\in\Gamma(\support_0)} \mathop{\sum_{i\in\support_0} \om_i |z_i|}_{e_{ij} \in E(\support_0)\backslash E'(\support_0)\quad} + \sum_{j\in\Gamma(\support_0)} \mathop{\sum_{i\in \support_0}}_{e_{ij} \in E'(\support_0)} \om_i |z_i|\,.
 \end{align}
 Equation \ref{eqn:wrnsp1} comes from the left $d$-regularity of the $(k,d,\e)$-lossless expander graph; while \eqref{eqn:wrnsp1b} is due to the definition of the collision set \eqref{eqn:collisionset}. Now we define 
\begin{equation}
 \label{eqn:gamma}
 \gm(j) = \{i \in \support_0 ~| ~e_{ij}\in E(\support_0)\backslash E'(\support_0)\}\,.
\end{equation}
Note that $|{\gm(j)}| = 1$ for each $j\in\Gamma(\support_0)$. Using this notation \eqref{eqn:wrnsp2} can be rewritten as
\begin{equation}
 \label{eqn:wrnsp3}
 d\|\z_{\support_0}\|_{\om,1} = \sum_{j\in\Gamma(\support_0)} \om_{\gm(j)} |z_{\gm(j)}| + \sum_{j\in\Gamma(\support_0)} \mathop{\sum_{i\in \support_0}}_{e_{ij} \in E'(\support_0)} \om_i |z_i|\,.
 \end{equation}
 Next we bound the first term on the right hand side of \eqref{eqn:wrnsp3}. The following follows from the fact that $\A$ is the adjacency matrix of a $(k,d,\e)$-lossless expander graph.
 \begin{align}
  (\A\z)_j = \sum_{i \in [N]} a_{ji} z_i = \mathop{\sum_{i\in [N]} z_i}_{e_{ij}\in E\quad} & = \sum_{l\geq 0} \mathop{\sum_{i\in \support_l} z_i}_{e_{ij}\in E\quad} \nonumber\\
  \label{eqn:elementofAz1}
  & = \mathop{\sum_{i\in \support_0} z_i}_{e_{ij}\in E(\support_0)\quad} + \sum_{l\geq 1} \mathop{\sum_{i\in \support_l} z_i}_{e_{ij}\in E\quad} \\
  \label{eqn:elementofAz2}
  & = \mathop{\sum_{i\in \support_0} z_i}_{e_{ij}\in E(\support_0)\backslash E'(\support_0)\quad} + \mathop{\sum_{i\in \support_0} z_i}_{e_{ij}\in E'(\support_0)\quad} + \sum_{l\geq 1} \mathop{\sum_{i\in \support_l} z_i}_{e_{ij}\in E\quad}\,.
 \end{align}
In \eqref{eqn:elementofAz1} we applied the splitting of the index set $[N]$; while \eqref{eqn:elementofAz2} is due to the definition of the collision set \eqref{eqn:collisionset}. With the ordering of the variable nodes we can use \eqref{eqn:gamma}, to rewrite \eqref{eqn:elementofAz2} as follows.
\begin{equation}
 \label{eqn:elementofAz3}
  (\A\z)_j = z_{\gm(j)} + \mathop{\sum_{i\in \support_0} z_i}_{e_{ij}\in E'(\support_0)\quad} + \sum_{l\geq 1} \mathop{\sum_{i\in \support_l} z_i}_{e_{ij}\in E\quad}\,.
\end{equation}
 
We multiply \eqref{eqn:elementofAz3} by $\om_{\gm(j)}$ and then we take absolute values to get
\begin{align}
 \om_{\gm(j)}|z_{\gm(j)}| &= \mathop{|\sum_{i\in \support_0} \om_{\gm(j)}z_i}_{e_{ij}\in E'(\support_0)\qquad\quad} + \sum_{l\geq 1} \mathop{\sum_{i\in \support_l} \om_{\gm(j)}z_i}_{e_{ij}\in E\qquad\quad} + ~\om_{\gm(j)}(\A\z)_j | \nonumber\\
 \label{eqn:elementofAz4}
 &\leq \mathop{\sum_{i\in \support_0} \om_{\gm(j)}|z_i|}_{e_{ij}\in E'(\support_0)\qquad\quad} + \sum_{l\geq 1} \mathop{\sum_{i\in \support_l} \om_{\gm(j)}|z_i|}_{e_{ij}\in E\qquad\quad} + ~\om_{\gm(j)}|(\A\z)_j| \\
 \label{eqn:elementofAz5}
 & \leq \mathop{\sum_{i\in \support_0} \om_i|z_i|}_{e_{ij}\in E'(\support_0)\qquad} + \sum_{l\geq 1} \mathop{\sum_{i\in \support_l} \om_i|z_i|}_{e_{ij}\in E\qquad} + ~\om_{\gm(j)}|(\A\z)_j| \,.
\end{align}
In \eqref{eqn:elementofAz4} we used the triangle inequality; while in \eqref{eqn:elementofAz5} we used the ordering of the entries of $\w$. Now we can bound  \eqref{eqn:wrnsp3} using the bound in \eqref{eqn:elementofAz5} as follows.
\begin{align}
\label{eqn:wrnsp4}
d\|\z_{\support_0}\|_{\om,1} & \leq \quad 2\sum_{j\in\Gamma(\support_0)} \mathop{\sum_{i\in \support_0} \om_i |z_i|}_{e_{ij} \in E'(\support_0)\quad} + \sum_{j\in\Gamma(\support_0)} \sum_{l\geq 1} \mathop{\sum_{i\in \support_l} \om_i|z_i|}_{e_{ij}\in E\qquad} + \sum_{j\in\Gamma(\support_0)} \om_{\gm(j)}|(\A\z)_j| \\
\label{eqn:wrnsp5}
& = \quad 2\sum_{e_{ij} \in E'(\support_0)}  \om_i |z_i| + \sum_{l\geq 1} \sum_{j\in\Gamma(\support_0)} \mathop{\sum_{i\in \support_l} \om_i|z_i|}_{e_{ij}\in E\qquad} + \sum_{j\in\Gamma(\support_0)} \om_{\gm(j)}|(\A\z)_j| \,.
\end{align}
In \eqref{eqn:wrnsp5} we used the fact that the double summation in the first term of the right hand side of  \eqref{eqn:wrnsp4} is equivalent to a single summation over all the edges in $E'(\support_0)$. Let the set of edges from vertex sets $\support_a$ and $\support_b$ be denoted as $E(\support_a : \support_b)$ . We upper bound the second term of \eqref{eqn:wrnsp5} in the following way.
\begin{align}
\sum_{l\geq 1} \sum_{j\in\Gamma(\support_0)} \mathop{\sum_{i\in \support_l} \om_i|z_i|}_{e_{ij}\in E\qquad} & \leq \sum_{l\geq 1} |E(\Gm(\support_0) : \support_{l})| \max_{i\in \support_{l}} \{\om_i|z_i|\}  \nonumber\\
\label{eqn:wrnsp5a1}
& \leq \sum_{l\geq 1} |E(\Gm(\support_0) : \support_{l})| \left(\frac{1}{k} \sum_{i\in \support_{l-1}} \om_i|z_i| \right)\,,
\end{align}
where  \eqref{eqn:wrnsp5a1} is due to the ordering assumption \eqref{eqn:orderingcondition}.
But $|E(\Gm(\support_0) : \support_{l})| = |\Gm(\support_0) \cap \Gm(\support_l) |$ and that $\Gm(\support_0) \cup \Gm(\support_l) = \Gm(\support_0 \cup \support_l)$ since $\support_0$ and $\support_l$ are disjoint. By the inclusion-exclusion principle $|\Gm(\support_0) \cap \Gm(\support_l) | = |\Gm(\support_0) | + |\Gm(\support_l) | - |\Gm(\support_0 \cup\support_l) |$ and by the expansion property of the $(k,d,\e)$-lossless expander graph $|\Gm(\support_0 \cup \support_l) | \geq (1-\e_{2k})d |\support_0 \cup \support_l |$. Thus we have 
\begin{align}
|E(\Gm(\support_0) : \support_{l})| & = |\Gm(\support_0) | + |\Gm(\support_l) | - |\Gm(\support_0 \cup\support_l) | \nonumber\\
\label{eqn:wrnsp5a2}
& \leq 2dk - 2(1-\e_{2k})dk = 2\e_{2k}dk\,,
\end{align}
where we upper bounded each of $|\Gm(\support_0) |$ and $|\Gm(\support_l) |$ by $dk$ since each node has at most $d$ neighbors.

Using this result we get the following upper bound for \eqref{eqn:wrnsp5a1}.
\begin{align}
\sum_{l\geq 1} \sum_{j\in\Gamma(\support_0)} \mathop{\sum_{i\in \support_l} \om_i|z_i|}_{e_{ij}\in E\qquad} & \leq \sum_{l\geq 1} 2\e_{2k}dk \left(\frac{1}{k} \|\z_{\support_{l-1}}\|_{\om,1} \right) \nonumber\\
& =  2\e_{2k}d \sum_{l\geq 1} \|\z_{\support_{l-1}}\|_{\om,1} \nonumber\\
\label{eqn:wrnsp5a3}
& \leq 2\e_{2k}d \|\z\|_{\om,1} \,.
\end{align}
For an upper bound of the last term of \eqref{eqn:wrnsp5} we proceed as follows. Note that this term is an inner product of two positive vectors hence we can use Cauchy-Schwarz inequality.
\begin{align}
\sum_{j\in\Gamma(\support_0)} \om_{\gm(j)}|(\A\z)_j| & \leq \sqrt{\sum_{j\in\Gamma(\support_0)} \om_{\gm(j)}^2} \sqrt{\sum_{j\in\Gamma(\support_0)} |(\A\z)_j|^2} \nonumber\\
\label{eqn:wrnsp5b1a}
& \leq \sqrt{\sum_{i\in\support_0} d\om_{i}^2} \|\A\z\|_2 \\
\label{eqn:wrnsp5b1}
&\leq \sqrt{ds} \|\A\z\|_1 \,.
\end{align}
In \eqref{eqn:wrnsp5b1a} we upper bounded by using the fact that each node has at most $d$ neighbors and that $|\Gm(\support_0)|\leq n$ to upper bound by the $\ell_2$ norm of $\A\z$. We upper bound the $\ell_2$ norm by the $\ell_1$ norm in \eqref{eqn:wrnsp5b1a} and used the bound $\sum_{i\in \support_0}\om_i^2  = \om(\support_0) \leq s$. 

Finally, we apply Lemma \ref{lem:collisionbound} to upper bound the first term of \eqref{eqn:wrnsp5} by $2\e_{2k}d\|\z_{\support_0}\|_{\om,1}$ (since $\e_{2k} \geq \e_k =: \e$). Then we use  \eqref{eqn:wrnsp5a3} and \eqref{eqn:wrnsp5b1} to respectively bound the second and the third terms of \eqref{eqn:wrnsp5} to get.
\begin{align}
\label{eqn:wrnsp6}
d\|\z_{\support_0}\|_{\om,1} & \leq 2\e_{2k}d\|\z_{\support_0}\|_{\om,1} + 2\e_{2k}d \|\z\|_{\om,1} +  \sqrt{ds} \|\A\z\|_1\\
\label{eqn:wrnsp7}
& = 4\e_{2k}d\|\z_{\support_0}\|_{\om,1} + 2\e_{2k}d \|\z_{\support_0^c}\|_{\om,1} +  \sqrt{ds} \|\A\z\|_1\,.
\end{align}
If we let $\support_0 = \support$ and the rearrange \eqref{eqn:wrnsp7} we have
\begin{equation}
\label{eqn:wrnspfinal}
\|\z_{\support}\|_{\om,1}  \leq \frac{2\e_{2k}}{1-4\e_{2k}} \|\z_{\support^c}\|_{\om,1} +  \frac{1}{\sqrt{d}(1-4\e_{2k})} \sqrt{s} \|\A\z\|_1 \,,
\end{equation}
which is the $\w$-RNSP \eqref{eqn:rnsp} with $\r$ and $\tau$ as in \eqref{eqn:wrnspconst}, hence concluding the proof.
\end{proof}

Based on Theorem \ref{thm:wrnsp} we provide reconstruction guarantees in the following theorem. 
\begin{theorem}
 \label{thm:errorbnd}
 Let $\A$ be the adjacency matrix of a $(k,d,\e)$-lossless expander graph with $\mathop{\e_{2k}<1/6}$. Given any $\x\in\RR^N$, if $\obs = \A\x + \noise$ with $\|\noise\|_1\leq \eta$, a solution $\sol$ of \eqref{eqn:wl1min} is an approximation of $\x$ with the following error bounds
 \begin{align}
  \label{eqn:errorwl1c}
  \|\sol - \x\|_{\om,1} \leq C_1 \sigma_s(\x)_{\om,1} + C_2 \sqrt{s} \eta, 
 \end{align}
 where the constants $C_1,C_2>0$ depend only on $d$ and $\e$.
\end{theorem}

Before we prove Theorem \ref{thm:errorbnd}, we state and prove a lemma, which is key to that proof. 
\begin{lemma}
 \label{lem:errorbnd}
 If $\A$ satisfies $\w$-RNSP \eqref{eqn:rnsp} with $\r<1$ and $\tau>0$, then given any $\x,\z\in\RR^N$ with $ \|\z\|_{\om,1} \leq  \|\x\|_{\om,1}$, we have
 \begin{align}
  \label{eqn:errorwl1a}
  \|\z - \x\|_{\om,1} \leq c_1 \sigma_s(\x)_{\om,1} + c_2 \sqrt{s}\|\A\left(\z - \x\right)\|_{1}, 
 \end{align}
 where the constants $c_1,c_2>0$ depend only on $d$ and $\e$.
\end{lemma}

\begin{proof}
 Let $\support$ with $|\support|\leq k$ such that $\om(\support)\leq s$. Then we have
 \begin{align}
\|\x\|_{\om,1} 	& \geq \|\z\|_{\om,1} \nonumber\\
		& = \|\z - \x + \x\|_{\om,1}\nonumber \\
 		& = \|\left(\z - \x + \x\right)_{\support}\|_{\om,1} + \|\left(\z - \x + \x\right)_{\support^c}\|_{\om,1}\nonumber\\
\label{eqn:prfthm01}		
 		& \geq \|\x_{\support}\|_{\om,1} - \|\left(\z - \x\right)_{\support}\|_{\om,1} + \|\left(\z - \x \right)_{\support^c}\|_{\om,1} - \|\x_{\support^c}\|_{\om,1} \\
\label{eqn:prfthm02}		
		& = \|\x\|_{\om,1} - 2\|\x_{\support^c}\|_{\om,1} - 2\|\left(\z - \x\right)_{\support}\|_{\om,1} + \|\z - \x \|_{\om,1}.
\end{align}
In \eqref{eqn:prfthm01} we used the triangle inequality; while in \eqref{eqn:prfthm02} we used the decomposability (separability) of the $\ell_{\om,1}$ norm, i.e. $\|(\cdot)_{\support}\|_{\om,1} + \|(\cdot)_{\support^c}\|_{\om,1} = \|\cdot\|_{\om,1}$. Simplifying and rearranging \eqref{eqn:prfthm02} gives
\begin{equation}
\label{eqn:prfthm03}	
2\|\x_{\support^c}\|_{\om,1} \geq \|\z - \x \|_{\om,1} - 2\|\left(\z - \x\right)_{\support}\|_{\om,1}.
\end{equation}
Now we are ready to use $\w$-RNSP1 to upper bound the last term of \eqref{eqn:prfthm03} by replacing ${\bf v}$ in \eqref{eqn:rnsp}  by $\z - \x$. Firstly \eqref{eqn:rnsp} can be rewritten as
\begin{equation}
 \label{eqn:lrnspmod}
  \|{\bf v}_{\mathcal{S}}\|_{\om,1} \leq \frac{\r}{1 + \r}\|{\bf v}\|_{\om,1} + \frac{\tau}{1 + \r}\sqrt{s}\|\A{\bf v}\|_{1}\,.
 \end{equation}
Therefore, using \eqref{eqn:lrnspmod} with ${\bf v} = \z - \x$, \eqref{eqn:prfthm03} becomes
 \begin{equation}
 \label{eqn:prfthm04}
 2\|\x_{\support^c}\|_{\om,1} \geq \|\z - \x \|_{\om,1} - \frac{2\r}{1 + \r}\|\z - \x\|_{\om,1} - \frac{2\tau}{1 + \r}\sqrt{s}\|\A\left(\z - \x\right)\|_{1}\,.
\end{equation}
Simplifying and rearranging \eqref{eqn:prfthm04} yields
\begin{equation}
\label{eqn:prfthm05}
\|\z - \x\|_{\om,1} \leq \frac{2(1 + \r)}{1 - \r}\|\x_{\support^c}\|_{\om,1} + \frac{4\tau}{1 - \r}\sqrt{s}\|\A\left(\z - \x\right)\|_{1}.
\end{equation}
Using the definition of $\sigma_s(\x)_{\om,1}$ from \eqref{eqn:stermapprox1} yields \eqref{eqn:errorwl1a} from \eqref{eqn:prfthm05} with 
\begin{equation}
\label{eqn:errorwl1const}
 c_1 = \frac{2(1 + \r)}{1 - \r}, \quad \mbox{and} \quad c_2 = \frac{4\tau}{1 - \r}. 
\end{equation}
This concludes the proof of the lemma.
\end{proof}


\begin{proof}{\em (Theorem \ref{thm:errorbnd})}
It can be easily seen that the theorem is a corollary of the Lemma  \ref{lem:errorbnd}. The recovery error bounds in \eqref{eqn:errorwl1c} follow from the error bounds in \eqref{eqn:errorwl1a} in Lemma \ref{lem:errorbnd} by replacing $\z$ with $\sol$ and using the triangle inequality to bound the following: $\|\A\left(\sol - \x\right)\|_{1} \leq \|\A\sol - \obs\|_{1} + \|\A\x - \obs\|_{1} \leq 2\eta$. 
Hence $C_1 = c_1$, and $C_2 = 2c_2$.
Finally, for $\A$ to satisfy the $\w$-RNSP with $\r<1$ and $\tau>0$ we require $\e_{2k} < 1/6$. 
\end{proof}

\subsection{Sample complexity}\label{sec:sample}
Here we derive sample complexities in terms of the weighted sparsity, $s$, of weighted sparse recovery using weighted $\ell_1$-minimization with sparse adjacency matrices of $(k,d,\e)$-lossless expander graphs. These sample complexity bounds are linear in the weighted sparsity of the signal and can be smaller than sample complexities of standard sparse recovery using unweighted $\ell_1$-minimization with and sparse adjacency matrices of $(k,d,\e)$-lossless expander graphs. Moreover, these results recover known results for the settings of a) uniform weights, b) polynomially growing weights, c) sparse recovery with prior support estimates, and d) known support. In particular, in the setting of {\em sparse recovery with prior support estimates}, depending on mild assumptions on the growth of the weights and how well is the support estimate aligned with the true support will lead to a reduction in sample complexity.  
The following derivations, without loss of generality, assume an ordering of the entries of the signal in order of magnitude such that $\support$ has the first $k$ largest in magnitude entries of the signal. 

\begin{theorem}
\label{thm:sampcomp_new}
Fix weights $\om_j \geq 1$. Suppose that $\gm > 0$ depending on the choice of weights, and $0\leq \d <1$.  Consider an adjacency matrix of a $(k,d,\e)$-lossless expander $\A\in\{0,1\}^{\dimensions}$, and a signal $\mathop{\x\in\RR^N}$ supported on $\support \subset [N]$ with $|\support|\leq k$ and  $\mathop{\sum_{i \in \support} \om_i^2 \leq s}$. Assume that noisy measurements are taken,  $\mathop{\obs = \A\x + \noise}$ with $\|\noise\|_1\leq \eta$ and $\e_{2k} < 1/6$. Then with probability at least $1-\d$, any solution $\sol$ of \eqref{eqn:wl1min} satisfies \eqref{eqn:errorwl1c}, if
 \begin{equation}
\label{eqn:sampcomp_new}
n = \bigO\left({s}/{(\e^2\gm)}\right), \quad \mbox{and} \quad d = \bigO\left({\e n}/{k}\right)\,.
\end{equation}
\end{theorem}

\begin{proof}
Theorem \ref{thm:errorbnd} guarantees that any solution $\sol$ of \eqref{eqn:wl1min} satisfies \eqref{eqn:errorwl1c} if the sensing matrix is an adjacency matrix of a $(k,d,\e)$-lossless expander with $\e_{2k} < 1/6$. Therefore, it suffice to prove the existence of such lossless expander graphs. The proof follows what has become a standard procedure for proving probabilistic existence of expander graphs \cite{berinde2008combining,bah2013vanishingly,foucart2013mathematical}. Consequently, we will skip some of the details of the proof. Let $G = ([N],[n],\edges)$ be a bipartite graph with $N$ left and $n$ right vertices. Let each vertex in $[N]$ have a regular degree $d$. We probabilistically construct the graph by picking each node in $[N]$ and connecting it to $d$ nodes in $[n]$ chosen uniformly at random. Then we ask that for any set $\support\subset [N], ~|\support|\leq k$ with $\om(\support)\leq s$, what is the probability of failure of the graph to expand on this set? Let this event be denoted by $\fail_k$, then
\begin{equation}
 \label{eqn:failevent}
 \fail_k := |\Gm(\support)| < (1 - \epsilon)dk.
\end{equation}
Therefore, we need to compute $\hbox{Prob}\left\{ \fail_k\right\}$ for our fixed $\support$ of size $k$, which we may not be able to do but an upper bound suffices. An upper bound is given by the following lemma proven in \cite{buhrman2002bitvectors}.
\begin{lemma}
\label{lem:tailbound}
Given a left $d$-regular bipartite graph, $G = ([N],[n],\edges)$, with expansion coefficient $\epsilon$, there exist a constant $\mu > 0$ such that for any $\mathcal{T} \subseteq [N]$ with $|\mathcal{T}| = t$, whenever $\numsam = \bigO\left(dt/\epsilon\right)$, we have
\begin{align*}
\hbox{\em Prob}\left\{ \fail_t\right\} \leq \left(\mu \cdot \frac{\epsilon \numsam}{dt}\right)^{-\epsilon d t}\,.
\end{align*}
\end{lemma}

\noindent For the bipartite graph $G$ to be an expander it has to expand on all sets $\support$ of $|\support|\leq k$. So we need the probability of failure on all set $\support$ of $|\support|\leq k$ which we can bound by a union bound as follows, if we denote this probability as $p$.
\begin{equation}
 \label{eqn:unionbound}
 p = \sum_{t=1}^k \binom{N}{t} \left(\mu \cdot \frac{\epsilon \numsam}{dt}\right)^{-\epsilon d t} \leq \sum_{t=1}^k e^{N\mathcal{H}(\frac{t}{N}) - \epsilon d t \log\left(\mu \cdot \frac{\epsilon \numsam}{dt}\right)} \,,
\end{equation}
where $\mathcal{H}(q) = -q\log q - (1-q) \log(1-q)$ is the Shannon entropy function in base $e$ logarithms and we bound the combinatorial term by a bound due to  \cite{cheney1966introduction}: $\binom{N}{Nq} \leq e^{N\cdot \mathcal{H}(q)}$, for $N \in \mathbb{Z}_{+}, ~q \in [0, 1]$ such that $qN \in \mathbb{Z}_{+}$.  
From Lemma \ref{lem:tailbound} the order notation implies that there exist a constant $C_1$ such that $d \geq C_1\e n/t$. Using this lower bound on $d$ and the bound $\mathcal{H}(x) < -x\log x + x$ found in \cite{bah2014bounds} we upper bound \eqref{eqn:unionbound} as follows:
\begin{equation}
 \label{eqn:unionbound2}
  p 	\leq \sum_{t=1}^k e^{N\left(-\frac{t}{N}\log\left(\frac{t}{N}\right) + \frac{t}{N}\right) - C_1\epsilon^2 n \log\left(\mu/ C_1\right)} = \sum_{t=1}^k e^{t\log\left({eN}/{t}\right) - C_2 \epsilon^2 n}\,, 
\end{equation}
where $C_2 = C_1\log\left(\mu/C_1\right)$. The function $t\log\left({eN}/{t}\right)$ is monotonically increasing in $t\geq 1$ and so its maximum occurs at $t=k$. Hence we can upper bound \eqref{eqn:unionbound2} as thus
\begin{equation}
 \label{eqn:unionbound3}
  p 	\leq k e^{k\log\left({eN}/{k}\right) - C_2 \epsilon^2 n}\,.
\end{equation}
For an expansion probability at least $1-\d$ we require $p \leq \d$, which hold if 
\begin{equation}
 \label{eqn:unionbound4a}
  k e^{k\log\left({eN}/{k}\right) - C_2 \epsilon^2 n} \leq \d, \quad \Rightarrow \quad \log(k) + k\log\left({eN}/{k}\right) - C_2 \epsilon^2 n \leq \log(\d)\,.
\end{equation}
We then choose weights such that $s/\gm$ is of the order of $k\log\left({N}/{k}\right)$. Examples of weight and $\gm$ choices are discussed in the next section. This concludes the proof.
\end{proof}

\begin{remark}
The proof requires that $s/\gm$ to be of the order of $k\log\left({N}/{k}\right)$, implying that we don't gain any reduction in the sample complexity in weighted sparse recovery over standard sparse recovery. This is an artifact of the proof technique. It is counter intuitive and the experiments (Section \ref{sec:numerics}) explicitly show the contrary with weighted sparse recovery having higher phase transitions (implying lower sampling rates) than standard sparse recovery. Nonetheless, it is interesting to express sample rates in terms of the weighted sparsity as this will guide the choice of weights.
\end{remark}

\subsection{Discussion on the choice of weights}\label{sec:disc}
Theorem \ref{thm:sampcomp_new} requires dependence of $\gm$ on the choice of weights, precisely it suffice to fix the weights such that $s/\gm$ of the order of $k\log\left({N}/{k}\right)$. Below we discuss the choice of weights and hence the choice of $\gm$, where these choices recovers existing results, similar to results shown in \cite{bah2015sample} for Gaussian sampling matrices.

\begin{itemize}
\item {\bf Uniform weights.}  In standard sparse recovery using unweighted $\ell_1$-minimization, the weights are $\om_i = 1$ for all $i\in [N]$.  This is a special case of Theorem \ref{thm:sampcomp_new} with $s = k$ and $\gm = \left(2 \log(N/k)\right)^{-1}$, thus recovering the known sample complexity results for standard sparse recovery with adjacency matrices of $(k,d,\e)$-lossless expander graphs, \cite{berinde2008combining,bah2013vanishingly,foucart2013mathematical}:
\begin{equation}
\label{eqn:sampcompstd}
n = \bigO \left(k \log(N/k)/\e^2\right), \quad \mbox{and} \quad d = \bigO \left(\log(N/k)/\e\right)\,.
\end{equation}

\item {\bf Polynomially growing weights.}  The idea of using polynomially growing weights was proposed in \cite{rauhut2015interpolation}, in the context of application to smooth function interpolation.  Precisely, the authors proposed weights $\om_i = i^{\alpha/2}$ for $\alpha \geq 0$. 
Using a number theoretic results due to \cite{shekatkar2012sum} we derive the following bound for a weighted $s$-sparse set $\support$ of cardinality $k$ (see details in Appendix): 
\begin{equation}
\label{eqn:numtheoreticbnd}
\sum_{i\in\support} \om_i^2 = \sum_{i\in\support} i^{\alpha} \leq (k+1)^{1+\alpha}\,.
\end{equation}
As such the weights (or $\alpha$) are chosen such that $(k+1)^{1+\alpha}$ of the order of $\gm k\log\left({N}/{k}\right)$ leading to \eqref{eqn:sampcomp_new}. Interestingly, if $\alpha = 0$ we recovery the standard sparse recovery result \eqref{eqn:sampcompstd} by choosing $\gm = \left(2 \log(N/k)\right)^{-1}$. Furthermore, since $(k+1)^{1+\alpha} \leq s ~\Rightarrow ~k \leq s^{\frac{1}{1+\alpha}}$ if we let $N = s^{1/\alpha}$ we have that $s/\gm$ is of the order of $\bigO\left(s^{\frac{1}{1+\alpha}}\log(s) \right)$, which is similar to a sample complexity suggested in \cite{rauhut2015interpolation} for dense random matrices.

\pagebreak
\item {\bf Sparse recovery with prior support estimates.}  In this case, we know $\suppest$ as an estimate of the true support $\support$ and typically we assign weights $\om_i = w \in [0,1]$ for $i \in \suppest$ and $\om_i = 1$ for $i\in \suppest^c$. Note that $\om_i \leq 1$ for all $i\in [N]$, contrary to the the setting of this work where $\om_i \geq 1$. Without loss of generality, we normalize by dividing by $w$ (for $w>0$) to get $\om_i = w_1 = 1$ for $i \in \suppest$ and $\om_i = w_2 = 1/w \geq 1$ for $i\in \suppest^c$ (if $w=0$, we divide by $w+\varepsilon$ with a small number $\varepsilon > 0$).  
The weighted cardinality of the support is
\begin{equation}
\label{s:prior}
\sum_{i\in \support} \om_i^2 = | \support \cap \suppest | + w^{-2} | \support \cap \suppest^c |\,.
\end{equation}
Like in \cite{rauhut2015interpolation} we can choose weights and $\gm$ as follows.
\begin{equation}
\label{eqn:kappa}
\om_i^2 \geq  \max\left\{ 1, 2\gm \log(j/s) \right\}\,,  \quad \text{with} \quad \gm = \min\left\{1,\left(2w^2 \log(N/s)\right)^{-1}\right\}\,.
\end{equation}
From Theorem \ref{thm:sampcomp_new} we have sample complexities $n = \bigO(s)$ if $\gm = 1$ and the more interesting case is when $\gm = \left(2w^2 \log(N/s)\right)^{-1}$, then
\begin{align}
\label{prior:support}
n = \bigO \left({s}/{\gm} \right) = \bigO \left( s w^2 \log(N/s) \right) =  \bigO \left(  \left( w^2| \support \cap \suppest | +  | \support \cap \suppest^c | \right)\log(N/s) \right)\,.
\end{align}
Let $| \suppest | = \beta | \support| = \beta k,$ and $| \support \cap \suppest | = \alpha | \suppest |$, where $\alpha, \beta \in [0,1]$.  Then the sampling bound in \eqref{prior:support} is bounded above by
\begin{equation}
\label{nusampleboundUS}
n = \bigO \left( \left(w^2 k + r\right) \log(N/k) \right),
\end{equation}
where $r = | \support^c \cap \suppest | + | \support \cap \suppest^c |$ represents the mismatch between the true and estimated supports.  Getting results similar to results in \cite{mansour2014recovery} for Gaussian matrices.

\item {\bf Known support.}  When the support of $\x$ coincides with the estimated support exactly, then $| \support \cap \suppest |  = s = k$ and $| \support \cap \suppest^c | = 0$, and the sample complexity becomes 
$$
n = \bigO(s) = \bigO(k),
$$
recovering the sample complexity of standard sparse recovery with known support.
\end{itemize}

\section{Experimental results} \label{sec:numerics}

In these experiments we consider the class of {\em weighted sparse signals} modeled in \cite{bah2015sample}.  Precisely, the probability for an index to be in the support of the signal is proportional to the reciprocal of the square of the weights assigned to that index. We also considered {\em polynomially growing weights}. In particular, we assign weights  $\om_j = j^{1/5}$ where the indices are ordered such that the support corresponds to the smallest in magnitude set of weights. 
The goal of the experiments was to compare the performance of weighted sparse recovery using weighted $\ell_1$-minimization and standard sparse recovery using unweighted $\ell_1$-minimization using both Gaussian sensing matrices and sensing matrices that are sparse binary adjacency matrices of expander graphs (hence forth referred to as expander matrices) in terms of a) sample complexity b) computational runtimes, and c) accuracy of reconstruction. The $m\times N$ Gaussian matrices have i.i.d. standard normal entries scaled by $\sqrt{m}$ while the expander matrices are generated by putting $d$ ones at uniformly at random locations in each column. We draw signals of dimension $N$ from the above mentioned model, where the nonzero values are randomly generated as scaled sums of Gaussian and uniformly random variables without any normalization.  We encode the signals using these matrices and add Gaussian white noise with noise level $\|\noise\|_2 \leq 10^{-6} =: \eta_2$ and define $\eta_1$ such that $\|\noise\|_1 \leq \eta_1$. 
For the weighted sparse reconstruction, we use \eqref{eqn:wl1min} with expanders and use a modified version of \eqref{eqn:wl1min}, replacing the $\ell_1$ by $\ell_2$ and $\eta_1$ by $\eta_2$ in the data fidelity term of \eqref{eqn:wl1min}, with Gaussian matrices; while the standard sparse reconstruction used 
\begin{equation}
\label{eqn:l1min}\tag{L1}
\min_{\z\in\RR^N} ~\|\z\|_{\om,1} \quad \mbox{subject to} \quad \|\A\z - \obs\|_p \leq \eta_p,
\end{equation}
with $p = 1$ for expanders and $p=2$ for Gaussian matrices. 

The following results are averaged over many realizations for each problem instance $(s,m,N)$. The dimension of the signal is $N = 2^{10}$. For the expander matrices we fixed $d = \lceil 2\log(N)\rceil$ and we vary the number of measurements $m$ such that $~m/N\in[\max(2d/N,0.05),0.35]$; and for each $m$ we vary the weighted sparsity of the supp($\x$), $\support$, such that $\om(\support)/m = s/m \in \left[{1}/{\min(m)},2.5\right]$. Then we record $k$ as the largest $|\support|$ for a given $s$. We consider a reconstruction successful if the recovery error in the $\ell_2$-norm is below $10\eta_1$ or $10\eta_2$ for expander or Gaussian matrices respectively and a failure otherwise. Then we compute the {\em empirical probabilities} as the ratio of the number of successful reconstructions to the number of realizations. 

\begin{itemize}
\item[a)] {\bf Sample complexities via phase transitions:}
We present below sample complexity comparisons using the phase transition framework in the phase space of $(s/m,m/N)$. Note that in all the figures we normalized (standardized) the values of $s/m$ in such a way that the normalized $s/m$ is between $0$ and $1$ for fair comparison.
The left panel of Figure \ref{fig:phasetransitions} shows phase transition curves in the form of contours of empirical probabilities of 50\% (solid curves) and 95\% (dashed curves) for expander and Gaussian matrices using either $\ell_1$ or $\ell_{\om,1}$ minimization. Both matrices have similar performance and by having larger area under the contours, weighted sparse recovery using \eqref{eqn:wl1min}, outperforms standard sparse recovery using \eqref{eqn:l1min}. The result in the left panel is further elucidated by the plots in the right panel of Figure \ref{fig:phasetransitions} and left panel of Figure \ref{fig:prob_runtime}. In the latter we show a snap shot for fixed $s/m = 1.25 $ and varying $m$ while in the former we show a snap shot for fixed $m/N = 0.1625 $ and varying $s$. Both plots confirm the comparative performance of expanders to Gaussian matrices and the superiority of weighted $\ell_1$ minimization over unweighted $\ell_1$ minimization.
\begin{figure}[h!]
\centering
\includegraphics[width=0.45\textwidth]{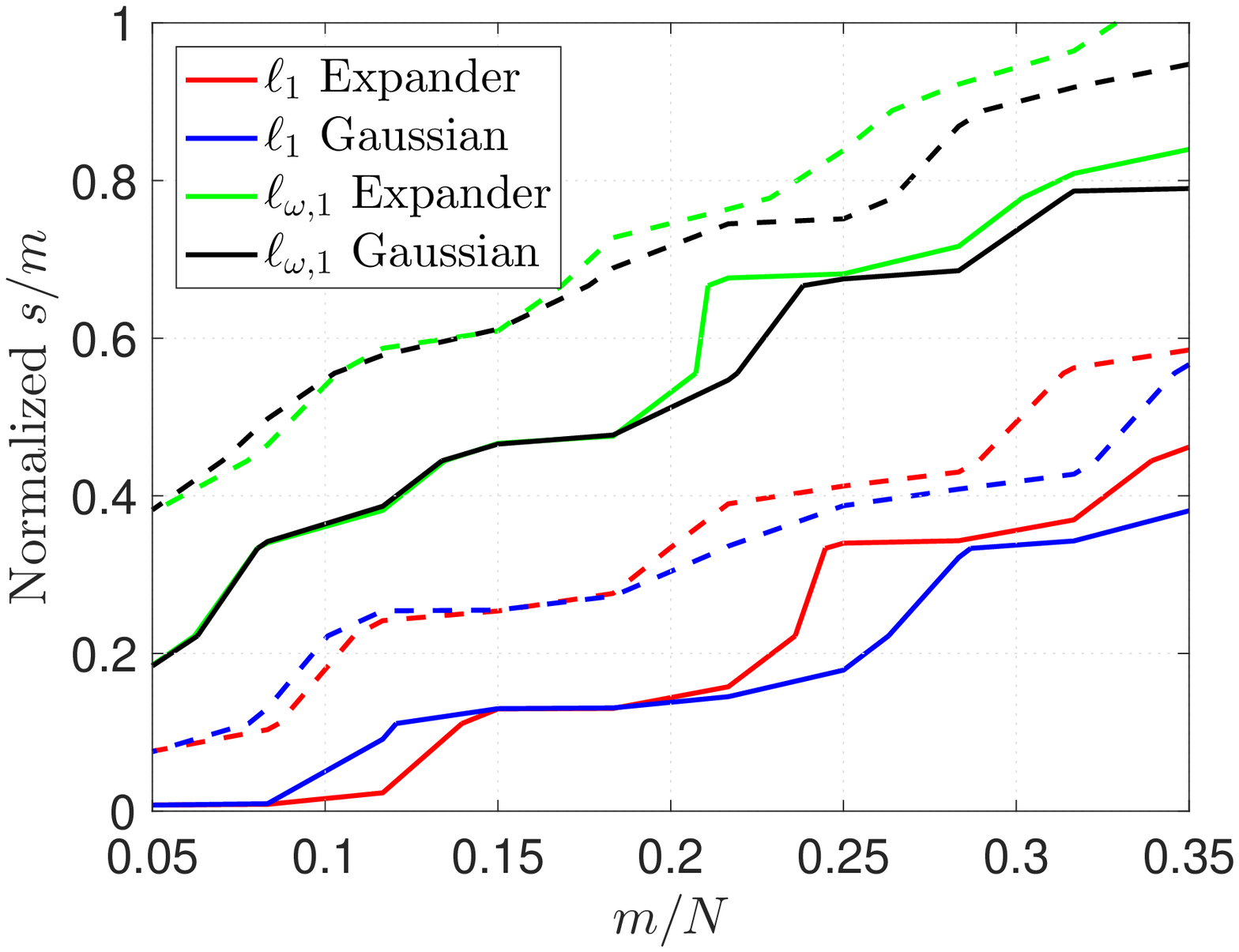} 
\includegraphics[width=0.45\textwidth]{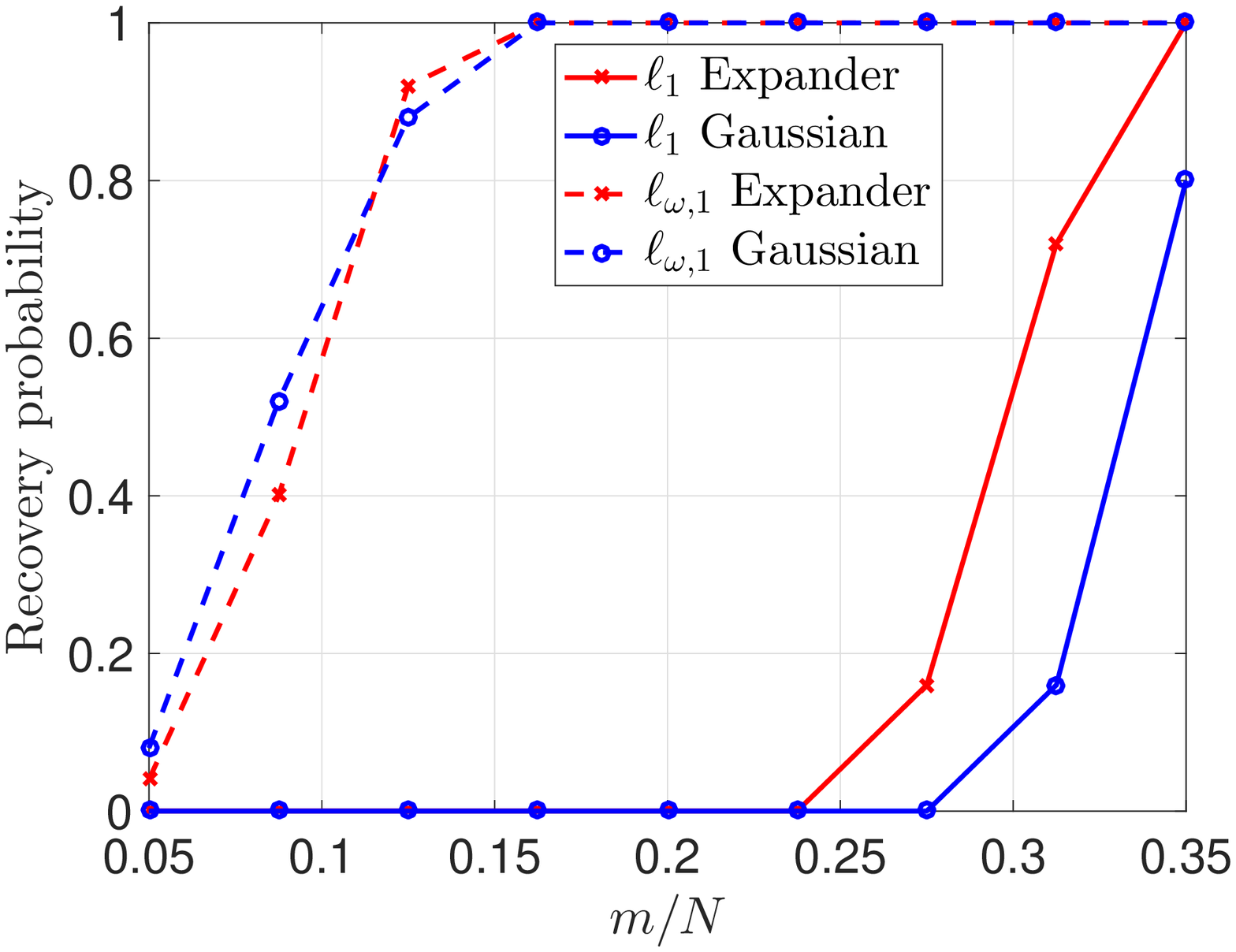} 
\caption{\small {\em Left panel}: Contour plots depicting phase transitions of 50\% and 95\% recovery probabilities  (dashed and solid curves respectively). {\em Right panel}: Recovery probabilities for a fixed $s/m = 1.25$ and varying $m$.}
\label{fig:phasetransitions}
\end{figure}

\item[b)] {\bf Computational runtimes:}
To compare runtimes we sum the generation time of $\A$ (Gaussian or expander), encoding time of the signal using $\A$, and the reconstruction time, with weighted $\ell_1$ minimization over unweighted $\ell_1$ minimization, and we average this over the number of realizations. In the right panel of Figure \ref{fig:prob_runtime} we plot average runtimes for varying $m/N$. This clearly shows that expanders have small runtimes.
\begin{figure}[h!]
\centering
\includegraphics[width=0.45\textwidth]{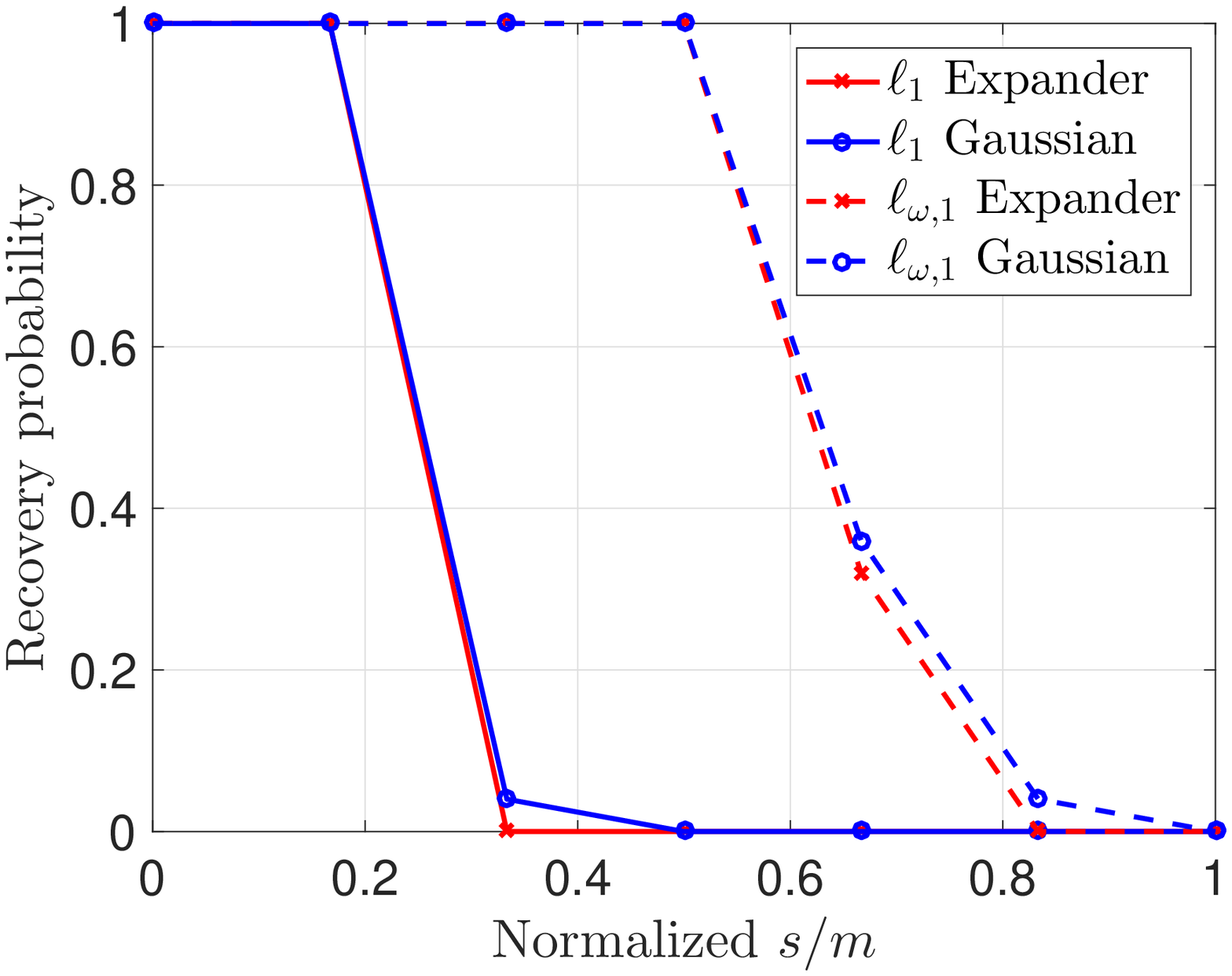} 
\includegraphics[width=0.45\textwidth]{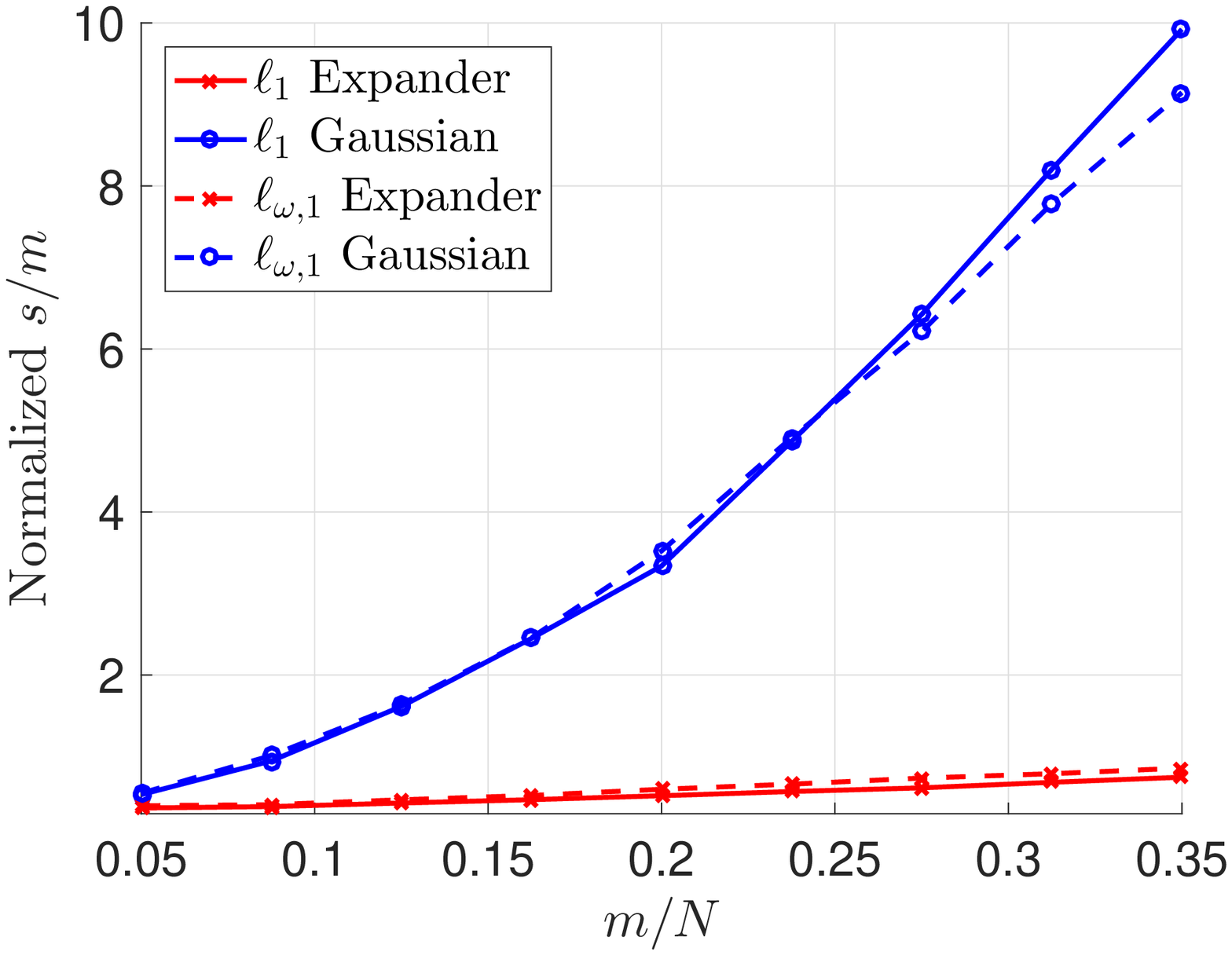} 
\caption{\small {\em Left panel}: Recovery probabilities for a fixed $m/N = 0.1625$ and varying $s$. {\em Right panel}: Runtime comparisons.}
\label{fig:prob_runtime}
\end{figure}

\item[c)] {\bf Accuracy of reconstructions:}
In Figure \ref{fig:lw1errors} we plot relative approximation errors in the $\ell_{\om,1}$ norm (top panel). The left panel are for a fixed $s/m = 1.25 $ and varying $m$ while in the right panel are for fixed $m/N = 0.1625 $ and varying $s$. In Figure \ref{fig:l2errors} we plot relative approximation errors in the $\ell_2$ norm. Similarly, the left panel are for a fixed $s/m = 1.25 $ and varying $m$ while in the right panel are for fixed $m/N = 0.1625 $ and varying $s$. In both Figures \ref{fig:lw1errors} and \ref{fig:l2errors} we see that weighted $\ell_1$ minimization converges faster with smaller number of measurements than unweighted $\ell_1$ minimization; but also we see that Gaussian sensing matrices have smaller approximation errors than the expanders. 
\begin{figure}[h!]
\centering
\includegraphics[width=0.45\textwidth]{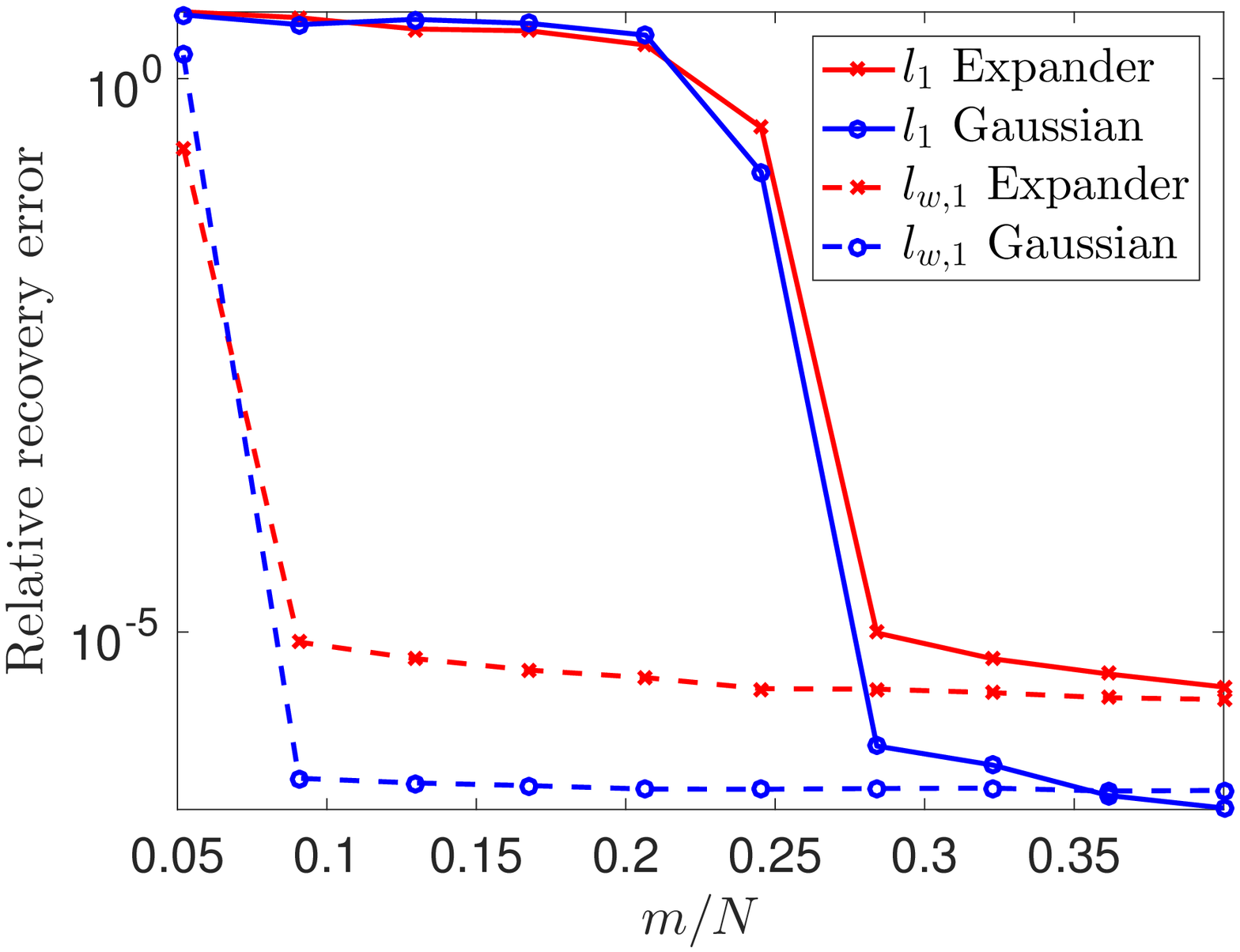} 
\includegraphics[width=0.45\textwidth]{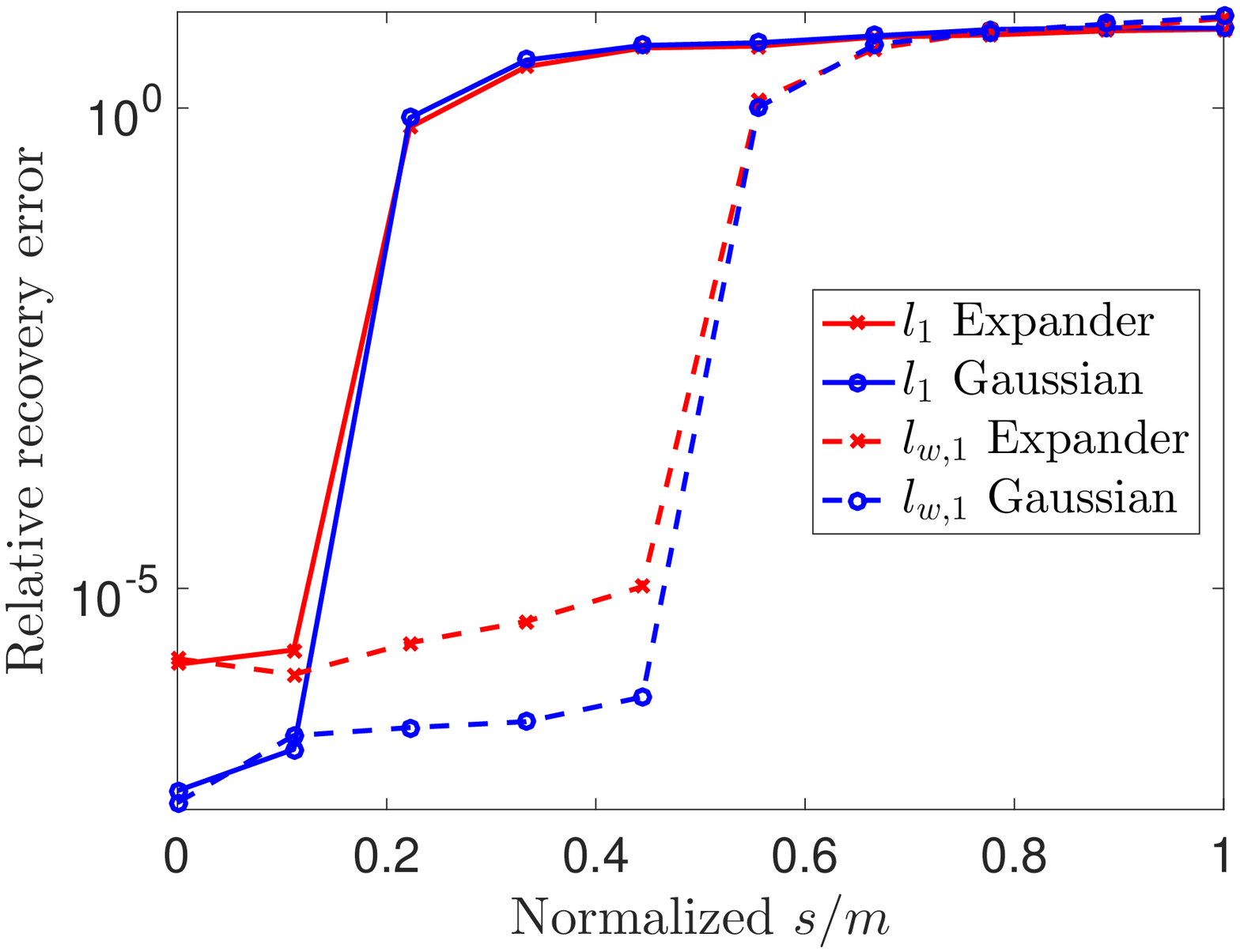} 
\caption{\small {\em Left panel}: Relative errors for a fixed $s/m = 1.25 $ and varying $m$. {\em Right panel}: Relative errors for a fixed $m/N = 0.1625$ and varying $s$.}
\label{fig:lw1errors}
\end{figure}

\begin{figure}[h!]
\centering
\includegraphics[width=0.45\textwidth]{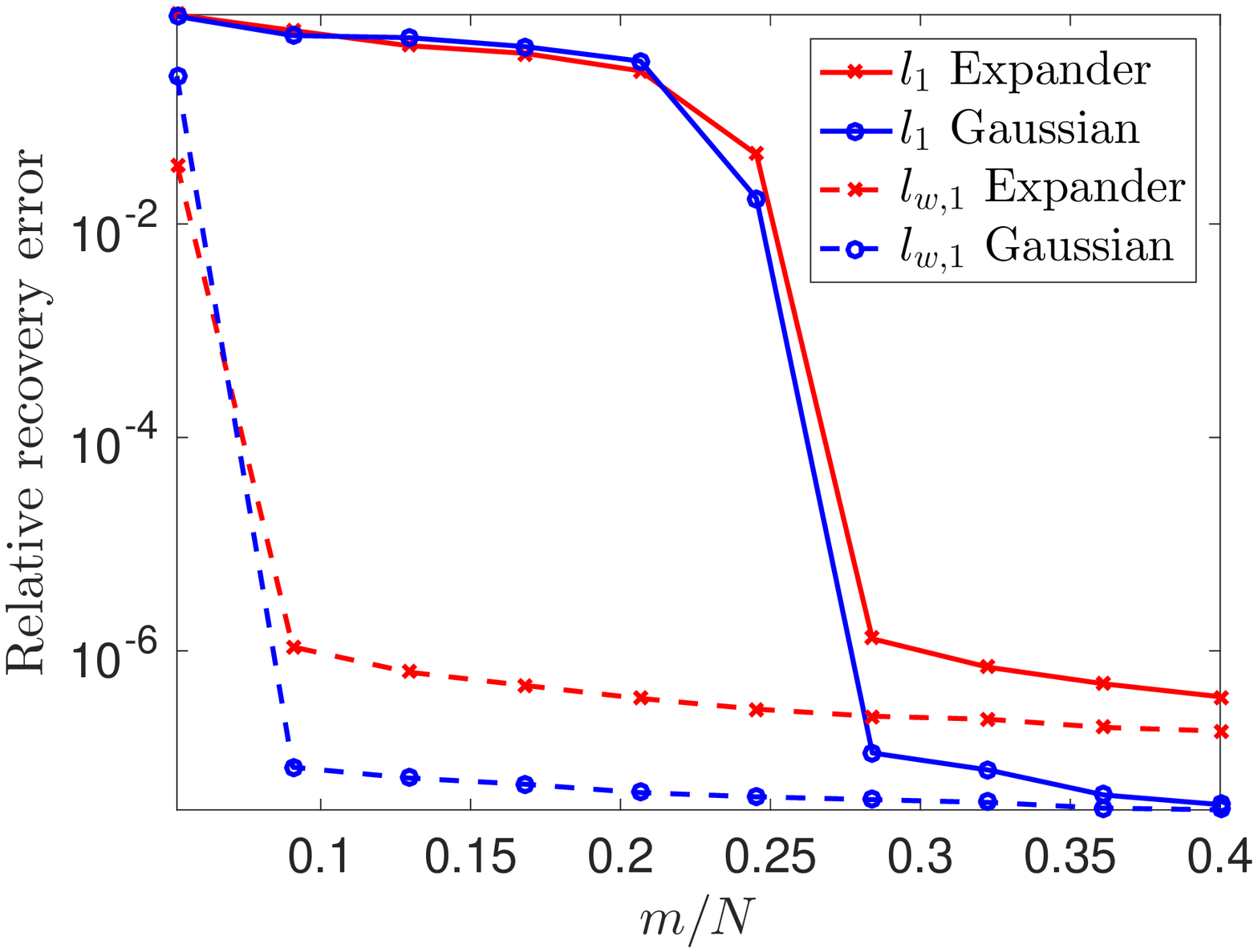} 
\includegraphics[width=0.45\textwidth]{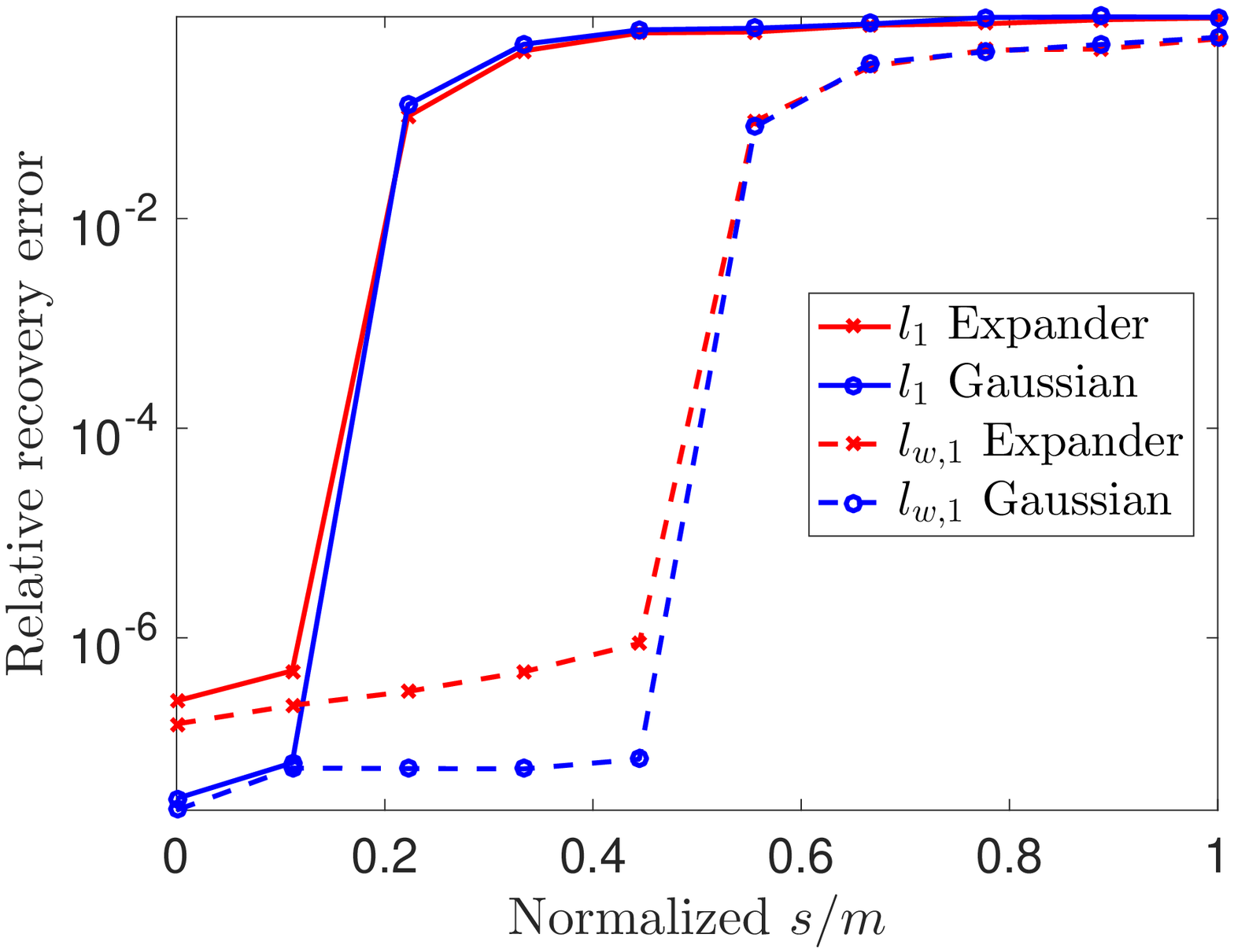} 
\caption{\small {\em Left panel}: Relative errors for a fixed $s/m = 1.25 $ and varying $m$. {\em Right panel}: Relative errors for a fixed $m/N = 0.1625$ and varying $s$.}
\label{fig:l2errors}
\end{figure}
\end{itemize}

\vspace{-0.7cm}
\section{Conclusion} \label{sec:concln}
We give the first rigorous error guarantees for weighted $\ell_1$ minimization with sparse measurement matrices and weighted sparse signals. The matrices are computationally efficient considering their fast application and low storage and generation complexities. The derivation of these error guarantees uses the weighted robust null space property proposed for the more general setting of weighted sparse approximations.  We also derived sampling rates for weighted sparse recovery using these matrices. These sampling bounds are linear in $s$ and can be smaller than sampling rates for standard sparse recovery depending on the choice of weights.  Finally, we demonstrate experimentally the validity of our theoretical results. Moreover, the experimental results show the computational advantage of expander matrices over Gaussian matrices.

\section*{Acknowledgements} 
The author is partially supported by Rachel Ward's NSF CAREER Grant, award \#1255631. The author would like to thank Rachel Ward, for reading the paper and giving valuable feedback that improved the manuscript. 

\vspace{+5mm}
\section{Appendix}

Here we prove the bound \eqref{eqn:numtheoreticbnd} by restating it as a lemma which we then prove.
\begin{lemma}
\label{lem:numtheoreticbnd}
Let $\support$ be a weighted $s$-sparse set of cardinality $k$ and $\om_i = i^{\a/2}$. Then
\begin{equation}
\label{eqn:numtheoreticbndlem}
\sum_{i\in\support} \om_i^2 = \sum_{i\in\support} i^{\alpha} \leq (k+1)^{1+\alpha}\,.
\end{equation}
\end{lemma}

\begin{proof}
The proof uses the following number theoretic results due to \cite{shekatkar2012sum}. We express $\a = 1/r$ where $r>1$ is a real number and state the results as a lemma.

\begin{lemma}
\label{lem:numhteory}
Let $r$ be a real number with $r\geq 1$ and $k$ be a positive integer. Then
\begin{equation}
\label{eqn:numtheory}
\sum_{i=1}^k i^{1/r} = \frac{r}{r+1} (k+1)^{\frac{1+r}{r}} - \frac{1}{2}(k+1)^{\frac{1}{r}} - \phi_k(r)\,,
\end{equation}
where $\phi_k$ is a function of $r$ with $k$ as a parameter. This function is bounded between $0$ and $1/2$.
\end{lemma}

\pagebreak
Using this results without proof (the interested reader is referred to \cite{shekatkar2012sum} for the proof) we have
\begin{align}
\label{eqn:kbound1}
\sum_{i\in\support} i^{\alpha} = \sum_{i\in\support} i^{1/r} & \leq \frac{r}{r+1} (k+1)^{\frac{1+r}{r}} - \frac{1}{2}(k+1)^{\frac{1}{r}} - \phi_k(r)\\
& \leq (k+1)^{1+\frac{1}{r}} - \frac{1}{2}(k+1)^{\frac{1}{r}} \\
& \leq \left(k+ \frac{1}{2}\right)(k+1)^{\frac{1}{r}} \\
& \leq (k+1)^{1+ \frac{1}{r}} = (k+1)^{1+ \a}\,,
\end{align}
as required, concluding the proof.
\end{proof}

\small
\bibliographystyle{plain}
\bibliography{bbtex}

\begin{thebibliography}{10}

\bibitem{adcock2015infinite2}
Ben Adcock.
\newblock Infinite-dimensional compressed sensing and function interpolation.
\newblock {\em arXiv preprint arXiv:1509.06073}, 2015.

\bibitem{adcock2015infinite}
Ben Adcock.
\newblock Infinite-dimensional $\ell_1$ minimization and function approximation
  from pointwise data.
\newblock {\em arXiv preprint arXiv:1503.02352}, 2015.

\bibitem{bah2013vanishingly}
Bubacarr Bah and Jared Tanner.
\newblock Vanishingly sparse matrices and expander graphs, with application to
  compressed sensing.
\newblock 2013.

\bibitem{bah2014bounds}
Bubacarr Bah and Jared Tanner.
\newblock Bounds of restricted isometry constants in extreme asymptotics:
  formulae for gaussian matrices.
\newblock {\em Linear Algebra and its Applications}, 441:88--109, 2014.

\bibitem{bah2015sample}
Bubacarr Bah and Rachel Ward.
\newblock The sample complexity of weighted sparse approximation.
\newblock {\em arXiv preprint arXiv:1507.06736}, 2015.

\bibitem{berinde2008combining}
R.~Berinde, A.C. Gilbert, P.~Indyk, H.~Karloff, and M.J. Strauss.
\newblock Combining geometry and combinatorics: A unified approach to sparse
  signal recovery.
\newblock In {\em Communication, Control, and Computing, 2008 46th Annual
  Allerton Conference on}, pages 798--805. IEEE, 2008.

\bibitem{bouchot2015compressed}
Jean-Luc Bouchot, Benjamin Bykowski, Holger Rauhut, and Christoph Schwab.
\newblock Compressed sensing {P}etrov-{G}alerkin approximations for parametric
  {PDEs}.
\newblock 2015.

\bibitem{buhrman2002bitvectors}
Harry Buhrman, Peter~Bro Miltersen, Jaikumar Radhakrishnan, and Srinivasan
  Venkatesh.
\newblock Are bitvectors optimal?
\newblock {\em SIAM Journal on Computing}, 31(6):1723--1744, 2002.

\bibitem{candes2006stable}
Emmanuel~J. Cand{\`e}s, Justin Romberg, and Terence Tao.
\newblock Stable signal recovery from incomplete and inaccurate measurements.
\newblock {\em Communications on pure and applied mathematics},
  59(8):1207--1223, 2006.

\bibitem{candes2008enhancing}
Emmanuel~J Cand{\`e}s, Michael~B Wakin, and Stephen~P Boyd.
\newblock Enhancing sparsity by reweighted $\ell_1$ minimization.
\newblock {\em Journal of Fourier analysis and applications}, 14(5-6):877--905,
  2008.

\bibitem{chandar2008negative}
V.t Chandar.
\newblock A negative result concerning explicit matrices with the restricted
  isometry property.
\newblock {\em preprint}, 2008.

\bibitem{cheney1966introduction}
Elliott~Ward Cheney.
\newblock Introduction to approximation theory.
\newblock 1966.

\bibitem{devore2007deterministic}
R.A. DeVore.
\newblock Deterministic constructions of compressed sensing matrices.
\newblock {\em Journal of Complexity}, 23(4):918--925, 2007.

\bibitem{duarte2008single}
Marco~F Duarte, Mark~A Davenport, Dharmpal Takhar, Jason~N Laska, Ting Sun,
  Kevin~E Kelly, Richard~G Baraniuk, et~al.
\newblock Single-pixel imaging via compressive sampling.
\newblock {\em IEEE Signal Processing Magazine}, 25(2):83, 2008.

\bibitem{flinth2015optimal}
Axel Flinth.
\newblock Optimal choice of weights for sparse recovery with prior information.
\newblock {\em arXiv preprint arXiv:1506.09054}, 2015.

\bibitem{foucart2013mathematical}
Simon Foucart and Holger Rauhut.
\newblock {\em A mathematical introduction to compressive sensing}.
\newblock Springer, 2013.

\bibitem{friedlander2012recovering}
Michael~P Friedlander, Hassan Mansour, Rayan Saab, and Oezguer Yilmaz.
\newblock Recovering compressively sampled signals using partial support
  information.
\newblock {\em Information Theory, IEEE Transactions on}, 58(2):1122--1134,
  2012.

\bibitem{haviv2015restricted}
Ishay Haviv and Oded Regev.
\newblock The restricted isometry property of subsampled fourier matrices.
\newblock {\em arXiv preprint arXiv:1507.01768}, 2015.

\bibitem{jacques2010short}
Laurent Jacques.
\newblock A short note on compressed sensing with partially known signal
  support.
\newblock {\em Signal Processing}, 90(12):3308--3312, 2010.

\bibitem{khajehnejad2009weighted}
Amin~M Khajehnejad, Weiyu Xu, Amir~S Avestimehr, and Babak Hassibi.
\newblock Weighted $\ell_1$ minimization for sparse recovery with prior
  information.
\newblock In {\em Information Theory, 2009. ISIT 2009. IEEE International
  Symposium on}, pages 483--487. IEEE, 2009.

\bibitem{mansour2014recovery}
Hassan Mansour and Rayan Saab.
\newblock Recovery analysis for weighted $\ell_1$-minimization using a null
  space property.
\newblock {\em arXiv preprint arXiv:1412.1565}, 2014.

\bibitem{mansour2011weighted}
Hassan Mansour and {\"O}zg{\"u}r Yilmaz.
\newblock Weighted-$\ell_1$ minimization with multiple weighting sets.
\newblock In {\em SPIE Optical Engineering $+$ Applications}, pages
  813809--813809. International Society for Optics and Photonics, 2011.

\bibitem{mendozaexpander}
Rodrigo Mendoza-Smith and Jared Tanner.
\newblock Expander $\ell_0$-decoding.

\bibitem{oymak2012recovery}
Samet Oymak, M~Amin Khajehnejad, and Babak Hassibi.
\newblock Recovery threshold for optimal weight $\ell_1$ minimization.
\newblock In {\em Information Theory Proceedings (ISIT), 2012 IEEE
  International Symposium on}, pages 2032--2036. IEEE, 2012.

\bibitem{peng2014weighted}
Ji~Peng, Jerrad Hampton, and Alireza Doostan.
\newblock A weighted $\ell_1$-minimization approach for sparse polynomial chaos
  expansions.
\newblock {\em Journal of Computational Physics}, 267:92--111, 2014.

\bibitem{rauhut2015interpolation}
Holger Rauhut and Rachel Ward.
\newblock Interpolation via weighted $\ell_1$ minimization.
\newblock {\em Applied and Computational Harmonic Analysis}, 2015.

\bibitem{shekatkar2012sum}
Snehal Shekatkar.
\newblock The sum of the $r$'th roots of first n natural numbers and new
  formula for factorial.
\newblock {\em arXiv preprint arXiv:1204.0877}, 2012.

\bibitem{vaswani2010modified}
Namrata Vaswani and Wei Lu.
\newblock Modified-{CS}: Modifying compressive sensing for problems with
  partially known support.
\newblock {\em Signal Processing, IEEE Transactions on}, 58(9):4595--4607,
  2010.

\bibitem{von2007compressed}
R~Von~Borries, Jacques~C Miosso, and Cristhian~M Potes.
\newblock Compressed sensing using prior information.
\newblock In {\em Computational Advances in Multi-Sensor Adaptive Processing,
  2007. CAMPSAP 2007. 2nd IEEE International Workshop on}, pages 121--124.
  IEEE, 2007.

\bibitem{xu2010breaking}
Weiyu Xu, M~Amin Khajehnejad, Amir~Salman Avestimehr, and Babak Hassibi.
\newblock Breaking through the thresholds: an analysis for iterative reweighted
  $\ell_1$ minimization via the grassmann angle framework.
\newblock In {\em Acoustics Speech and Signal Processing (ICASSP), 2010 IEEE
  International Conference on}, pages 5498--5501. IEEE, 2010.

\end{thebibliography}

\end{document}